\documentclass[10pt,reqno]{amsart}
\usepackage{amsmath,amsopn,amssymb,amsthm}
\usepackage[cp1251]{inputenc}
\usepackage[english]{babel}
\usepackage[pagebackref,breaklinks=true,colorlinks=true,linkcolor=blue,citecolor=blue,urlcolor=blue]{hyperref}
\usepackage{graphicx}%
\usepackage{color}

\newtheorem{theorem}{Theorem}[section]

\newtheorem{corollary}[theorem]{Corollary}

\newtheorem{definition}[theorem]{Definition}

\newtheorem{lemma}[theorem]{Lemma}

\newtheorem{remark}[theorem]{Remark}

\renewenvironment{proof}[1][Proof]{\noindent\textbf{#1.} }{\ \rule{0.5em}{0.5em}}

\begin{document}
\title[Metric operator and g.o. property for standard homogeneous Finsler metric]{Metric operator and geodesic orbit property for a standard homogeneous Finsler metric}
\author{Lei Zhang}
\address[Lei Zhang]
{School of Mathematical Sciences and Information Science, Yantai University, Yantai City 264005, P.R. China}
\email{546502871@qq.com}
\author{Ming Xu}
\address[Ming Xu] {Corresponding author, School of Mathematical Sciences,
Capital Normal University,
Beijing 100048,
P.R. China}
\email{mgmgmgxu@163.com}
\date{}
\maketitle
\begin{abstract}
In this paper, we introduce the metric operator for a compact homogeneous Finsler space, and use it to investigate the geodesic orbit property. We define the notion of standard homogeneous $(\alpha_1,\cdots,\alpha_s)$-metric which generalizes the notion of standard homogeneous $(\alpha_1,\alpha_2)$-metric. We classify all connected simply connected homogeneous manifold $G/H$ with a compact connected simple Lie group $G$ and two irreducible summands in its isotropy representation, such that there exists a standard homogeneous $(\alpha_1,\alpha_2)$-metric which is
g.o. but not naturally reductive on $G/H$. We also prove that on a generalized Wallach space which is
not a product of three symmetric spaces, any standard homogeneous $(\alpha_1,\alpha_2,\alpha_3)$-metric $F$ with respect to the canonical decomposition is g.o. on $G/H$ if and only if $F$ is a normal homogeneous Riemannian metric.


\end{abstract}
\section{Introduction}

A homogeneous Finsler space $(G/H,F)$ is called geodesic orbit (or g.o. in short), if any
geodesic (of positive constant speed) is the orbit of the one-parameter subgroup generated
by a Killing vector field $X\in \mathfrak{g} = Lie(G)$. This notion was first introduced in Riemannian geometry by
O. Kowalski and L. Vanhecke in 1991 \cite{KV}. There are many research works on this subject. See \cite{AW,BN2020,DA,DN,Gordon,GN} and the references therein for some recent progress. Meanwhile, the geodesic orbit property was studied in Finsler geometry \cite{YD}. It is well known that weakly symmetric Finsler spaces, which included all globally symmetric Finsler spaces, are g.o. \cite{D1}. Normal homogeneous Finsler spaces \cite{XD} and $\delta$-homogeneous Finsler spaces \cite{XZ} are also g.o., and they have many interesting curvature properties, for example, their flag curvatures are non-negative and their S-curvatures vanish.

In this paper, we only consider the homogeneous manifold $G/H$ with a compact connected semi simple $G$. Let $\mathfrak{g}$ and $\mathfrak{h}$ be the Lie algebras of $G$ and $H$ respectively, and we choose the $\mathrm{Ad}(G)$-invariant inner product $Q(\cdot,\cdot)=-B(\cdot,\cdot)$ on $\mathfrak{g}$, where $B(\cdot,\cdot)$ is the Killing form.
With respect to $Q$, we have an orthogonal reductive decomposition $\mathfrak{g}=\mathfrak{h}+\mathfrak{m}$ (here the reductiveness implies $[\mathfrak{h}, \mathfrak{m}]\subseteq \mathfrak{m}$). We can identify $\mathfrak{m}$ with the tangent space $T_o(G/H)$ at the origin $o=eH$, such that the isotropy representation coincides with the $\mathrm{Ad}(H)$-action on $\mathfrak{m}$.

Any homogeneous Riemannian metric $\mathrm{g}$ on $G/H$ is determined by some positive $Q$-symmetric $Ad(H)$-equivariant metric operator $A:\mathfrak{m}\rightarrow \mathfrak{m}$ by
the formula $\mathrm{g}_{eH}(X,Y)=Q(AX,Y)$, $\forall X,Y\in \mathfrak{m}$. Proposition 1 in \cite{DA} and Proposition 2 in \cite{S} show that $(G/H,\mathrm{g})$ is g.o. if and only if for any $X\in \mathfrak{m}$, there exists $Z\in \mathfrak{h}$ such that
$[X+Z,AX]=0$,
and $(G/H,\mathrm{g})$ is naturally reductive with respect to the given orthogonal reductive decomposition if and only if $
  [X,AX]=0$, $\forall X\in\mathfrak{m}$.

Now we generalize the notion of metric operator to homogeneous Finsler manifold.
There is a one-to-one correspondence between the $G$-invariant Finsler metric $F$ on $G/H$ and the $\mathrm{Ad}(H)$-invariant Minkowski norm, which is still denoted by $F$ for simplicity \cite{D1,DH1}. Let $g_y$ with $y\in \mathfrak{m}\backslash\{0\}$ be the fundamental tensor of the Minkowski norm $F$. Then we define the {\it metric operator} of $(G/H,F)$ as
\begin{equation*}
  A_y:\mathfrak{m}\rightarrow \mathfrak{m},\ g_y(u,v)=Q(A_y(u),v),\ \forall y\in \mathfrak{m}\backslash \{0\},u,v\in \mathfrak{m}.
\end{equation*}
Obviously, the metric operator $A_y$ is positive definite and $Q$-symmetric for each $y\in\mathfrak{m}\backslash\{0\}$. Generally speaking, it depends on $y$ and it is not $\mathrm{Ad}(H)$-invariant, but we  can still use it to describe the g.o. and naturally reductive properties (see Lemma \ref{lemma3.4} and Lemma \ref{lemma 3.6}).

In this paper, we use this Finsler metric operator to study the g.o. property.
Our first main theorem generalizes Theorem 2 in \cite{CN} to Finsler geometry.

\begin{theorem}\label{theorem 1.1}
 Let $G/H$ be a connected simply connected homogeneous manifold such that $G$ is a compact connected simple Lie group, and the isotropy representation is the sum of two irreducible summands. Suppose that $G/H$ admits a standard homogeneous $(\alpha_1,\alpha_2)$-metric which is g.o. but not naturally reductive. Then there exists a compact subgroup $K$ such that $H\subset K\subset G$, $\dim H<\dim K<\dim G$, $G/K$ is symmetric, and the triple $(H,K,G)$ coincides with one of the following in the Lie algebraic level:
\begin{enumerate}
  \item[(1)] $G_2\subset Spin(7)\subset Spin(8)$;
  \item[(2)] $SO(2)\times G_2\subset SO(2)\times SO(7)\subset SO(9)$;
  \item[(3)] $U(k)\subset SO(2k)\subset SO(2k+1)$ for $k\geq 2$;
  \item[(4)] $SU(2r+1)\subset U(2r+1)\subset SO(4r+2)$ for $r\geq 2$;
  \item[(5)] $Spin(7)\subset SO(8)\subset SO(9)$;
  \item[(6)] $SU(m)\times SU(n)\subset S(U(m)U(n))\subset SU(m+n)$ for $m> n\geq 1$;
  \item[(7)] $Sp(n)U(1)\subset S(U(2n)U(1))\subset SU(2n+1)$ for $n\geq 2$;
  \item[(8)] $Sp(n)U(1)\subset Sp(n)\times Sp(1)\subset Sp(n+1)$ for $n\geq 1$;
  \item[(9)] $Spin(10)\subset Spin(10)SO(2)\subset E_6$.
\end{enumerate}
Conversely, if $G/H$ coincides with one in the list in the Lie algebraic level, then
any standard homogeneous $(\alpha_1,\alpha_2)$-metric $F$ on $G/H$ is g.o..
\end{theorem}

The notion of $(\alpha_1,\alpha_2)$-metric was introduced in \cite{DX2016}. In \cite{ZX}, we defined the notion of standard
homogeneous $(\alpha_1,\alpha_2)$-metric, and in particular, studied those which have exactly two irreducible summands in their isotropy representations. Theorem \ref{theorem 1.1} provides a  classification for these homogeneous manifolds.

Standard homogeneous $(\alpha_1,\alpha_2)$-metric, with respect to an $\mathrm{Ad}(H)$-invariant $Q$-orthogonal decomposition $\mathfrak{m}=\mathfrak{m}_1+\mathfrak{m}_2$,
is an analog of the standard deformation of normal homogeneous metrics in Riemannian geometry. This notion can be further generalized when there are more summands in $\mathfrak{m}$. Let $G/H$ be the compact homogeneous manifolds mentioned above, and $\mathfrak{m}=\mathfrak{m}_1+\cdots+\mathfrak{m}_s$ be an $\mathrm{Ad}(H)$-invariant decomposition. Then we call $(G/H,F)$ a {\it homogeneous $(\alpha_1,\cdots,\alpha_s)$-metric} with respect to this given decomposition, if $F$ can be presented as $F=\sqrt{L(\alpha_1^2,\cdots,\alpha_s^2)}$, i.e., the Minkowski norm $F$ on $\mathfrak{m}$ has a block-diagonal linear $O(\mathfrak{m}_1)\times\cdots\times O(\mathfrak{m}_s)$-symmetry. In particular, when the decomposition $\mathfrak{m}=\mathfrak{m}_1+\cdots+\mathfrak{m}_s$ is $Q$-orthogonal, and
each orthogonal group $O(\mathfrak{m}_i)$ is with respect to $Q|_{\mathfrak{m}_i\times\mathfrak{m}_i}$, we call $(G/H,F)$ a {\it standard homogeneous $(\alpha_1,\cdots,\alpha_s)$-space} (or {\it standard homogeneous Finsler space} in short).
See Section \ref{section-3-3} for more details.

For example, any (compact) generalized Wallach space $G/H$ admits a canonical decomposition, i.e.,
an  $\mathrm{Ad}(H)$-invariant $Q$-orthogonal decomposition $\mathfrak{m}=\mathfrak{m}_1+\mathfrak{m}_2+\mathfrak{m}_3$,
with $[\mathfrak{m}_i,\mathfrak{m}_i]\subset\mathfrak{h}$ for each $\mathfrak{h}$ \cite{Niko1}.
So this $G/H$ admits standard homogeneous $(\alpha_1,\alpha_2,\alpha_3)$-metrics with respect to this canonical decomposition. We prove

\begin{theorem}\label{theorem 1.2}
Let $G/H$ be a generalized Wallach space, and $F$ a $G$-invariant standard homogeneous $(\alpha_1,\alpha_2,\alpha_3)$-metric, with respect to the canonical decomposition. Suppose that $G/H$ is not a product of three symmetric spaces and $F$ is g.o. on $G/H$, then $F$ is a normal homogeneous Riemannian metric.
\end{theorem}

Because the Wallach spaces $SU(3)/T^2$, $Sp(3)/Sp(1)^3$ and $F_4/Spin(8)$ are generalized Wallach spaces, Theorem \ref{theorem 1.2}
refines Theorem 6.2 in \cite{ZX} (see Corollary \ref{cor-1}).
It reveals the phenomena that sometimes the g.o. property is more algebraic than geometric, i.e., it depends on the homogeneous manifold, not the invariant metric.

This paper is scheduled as follows. In Section 2, we summarize some basic knowledge in homogeneous Finsler geometry. In Section 3, we introduce the notions of Finslerian metric operator and standard homogeneous $(\alpha_1,\cdots,\alpha_s)$-metric.
In Section 4, we discuss the metric operator of a standard homogeneous $(\alpha_1,\alpha_2)$-metric and prove Theorem 1.1. In Section 5, we discuss the standard homogeneous $(\alpha_1,\alpha_2,\alpha_3)$-metric on a generalized Wallach space and prove Theorem 1.2.

\section{Naturally reductive and geodesic orbit Finsler spaces}

Throughout the paper, we work with a homogeneous Finsler space $(G/H,F)$ with a compact connected semi simple $G$. On its Lie algebra $\mathfrak{g}$, we fix the $\mathrm{Ad}(G)$-invariant inner product
$Q(\cdot,\cdot)=-B(\cdot,\cdot)$ where $B$ is the Killing form of $\mathfrak{g}$. With respect to $Q$, we have the orthogonal reductive decomposition $\mathfrak{g}=\mathfrak{h}+\mathfrak{m}$.
The $G$-invariant Finsler metric $F$ is one-to-one determined by $F=F(eH,\cdot)$,
which is an arbitrary $\mathrm{Ad}(H)$-invariant Minkowski norm on $\mathfrak{m}=T_{eH}(G/H)$. we denote by $g_y(\cdot,\cdot)$ the fundamental tensor of the Minkowski norm $F$. See \cite{D1,DH1} for more details.

We call $(G/H,F)$ {\it $G$-naturally reductive} with respect to the orthogonal reductive decomposition
$\mathfrak{g}=\mathfrak{h}+\mathfrak{m}$ if for each nonzero $u\in \mathfrak{m}$, $c(t)=\exp tu \cdot H$ is a geodesic, or equivalently speaking, $g_u(u,[u,v]_\mathfrak{m})=0$, $\forall v\in\mathfrak{m}$,
in which the subscript $\mathfrak{m}$ means projecting to $\mathfrak{m}$ with respect to the given reductive decomposition \cite{DH2}. Notice that
a different definition for naturally reductive Finsler space was proposed in \cite{L}.
It turns out that the definition in \cite{DH1} is more convenient and both are equivalent \cite{ZYD2023}.

We call $(G/H,F)$ {\it geodesic orbit} (or {\it g.o.} in short) if any geodesic $c(t)$ is homogeneous, i.e., $c(t)=\exp tX \cdot x$ for some $X\in \mathfrak{g}$.
The following equivalent descriptions for the Finsler g.o. property is well known \cite{X,YD}.

\begin{lemma}\label{lemma 2.1}
 Let $(G/H,F)$ be a homogeneous Finsler space, with a reductive decomposition
 $\mathfrak{g} = \mathfrak{h} + \mathfrak{m}$, and denote $[\cdot,\cdot]_\mathfrak{m}$ the $\mathfrak{m}$-factor in the bracket operation $[\cdot,\cdot]$.
Then the following are equivalent:
\begin{enumerate}
  \item[(1)] $F$ is $G$-g.o.;
  \item[(2)] for any $x\in M$, and any nonzero $y\in T_xM$, we can find a Killing vector field $X\in \mathfrak{g}$ such that $X(x)=y$ and $x$ is a critical point for the function $f(\cdot)=F(X(\cdot))$;
  \item[(3)] for any nonzero vector $u\in \mathfrak{m}$, there exists $u'\in \mathfrak{h}$ such that
  $g_u([u+u',\mathfrak{m}]_{\mathfrak{m}},u)=0$;
  \item[(4)] the spray vector field $\eta(\cdot):\mathfrak{m}\backslash \{0\}\rightarrow \mathfrak{m}$ is tangent to the $Ad(H)$-orbits.
\end{enumerate}
\end{lemma}

Here the {\it spray vector field} $\eta(\cdot)$ was defined by L. Huang \cite{H1}, and it satisfies
$g_y(\eta(y),u)=g_y(y,[u,y]_{\mathfrak{m}})$, $\forall u\in \mathfrak{m}$.

\section{Metric operator and standard homogeneous Finsler metric}

\subsection{Metric operator for a homogeneous Finsler metric}

\begin{definition}
For each nonzero vector $y\in\mathfrak{m}$,
the metric operator of $(G/H,F)$ is the $g_y$-symmetric, positive definite linear endomorphism $A_y$ on $\mathfrak{m}$ determined by $g_y(u,v)=Q(A_y(u),v)$, $\forall u,v\in \mathfrak{m}$.
 \end{definition}

In particular, when $F$ is Riemannian, $A_y$ is irrelevant to $y$ and coincides with the Riemannian metric operator. Notice that each $A_y$ may not be $\mathrm{Ad}(H)$-equivariant as in Riemannian geometry, but the $\mathrm{Ad}(H)$-invariance of the Minkowski norm $F$ can still imply
\begin{lemma}\label{Proposition 3.2}The metric operator of $(G/H,F)$ satisfies:
\begin{enumerate}
\item[(1)] $\mathrm{Ad}(g)\circ A_y\circ \mathrm{Ad}(g)^{-1}=A_{\mathrm{Ad}(g)y}$, $\forall g\in H$, $ y\in\mathfrak{m}\backslash\{0\}$;
\item[(2)]
$[A_y,\mathrm{ad}(w)](y)=0$, $\forall y\in\mathfrak{m}\backslash \{0\},w\in \mathfrak{h}$.
\end{enumerate}
\end{lemma}
\begin{proof}
Since the Minkowski norm $F$ on $\mathfrak{m}$
 has the $Ad(H)$-invariance, so does its fundamental tensor, i.e.,
\begin{equation}\label{002}
g_{\mathrm{Ad}(g)y}(\mathrm{Ad}(g)u,\mathrm{Ad}(g)v)=g_y(u,v),\quad
\forall g\in H, y\in\mathfrak{m}\backslash\{0\}, u,v\in\mathfrak{m}.
\end{equation}
So we have for any $g\in H$ and $u,v\in\mathfrak{m}$,
\begin{eqnarray*}
& &Q(A_{\mathrm{Ad}(g)y}(\mathrm{Ad}(g)u),\mathrm{Ad}(g)v)
=g_{\mathrm{Ad}(g)y}(\mathrm{Ad}(g)u,\mathrm{Ad}(g)v)\\&=&g_y(u,v)
=Q(A_\mathrm{y}(u),v)=Q(\mathrm{Ad}(g)(A_\mathrm{y}(u)),\mathrm{Ad}(g)v)\\
&=&Q((\mathrm{Ad}(g)\circ A_\mathrm{y}\circ\mathrm{Ad}(g)^{-1})(\mathrm{Ad}(g)u),\mathrm{Ad}(g)v),
\end{eqnarray*}
which proves (1).

Input $g=\exp tw$ with $w\in \mathfrak{h}$ and $u=y$ into (\ref{002}),
differentiate it with respect to $t$, and take $t=0$,
we get
\begin{equation}
g_y([w,y],v)+g_y(y,[w,v])+2C_{y}([w,y],y,v)=g_y([w,y],v)+g_y(y,[w,v])=0.\label{004}
\end{equation}
So for any $w\in\mathfrak{h}$, $y\in\mathfrak{m}\backslash\{0\}$ and $v\in\mathfrak{m}$,
\begin{eqnarray*}
& &Q((A_y\circ \mathrm{ad}(w))(y),v)=g_y([w,y],v)=-g_y(y,[w,v])\\
&=&-Q(A_y(y),[w,v])
=Q([w,A_y(y)],v)=Q((\mathrm{ad}(w)\circ A_y)(y),v),
\end{eqnarray*}
i.e., $[A_y,\mathrm{ad}(w)](y)=(A_y\circ \mathrm{ad}(w)-\mathrm{ad}(w)\circ A_y)(y)=0$, which proves
(2).
\end{proof}

We can use the metric operator to describe
the naturally reductive and  geodesic orbit properties in Finsler geometry.
For the g.o. property, we have
\begin{lemma}\label{lemma3.4}
$(G/H,F)$ is g.o. if and only if for any nonzero vector $u\in\mathfrak{m}$, there exists $u'\in\mathfrak{h}$, such that $[u'+u, A_u(u)]=0$.
\end{lemma}

\begin{proof}First, we assume that $(G/H,F)$ is g.o.. Let $u$ be any vector in $\mathfrak{m}\backslash\{0\}$. By Lemma \ref{lemma 2.1}, there exists
$u'\in \mathfrak{h}$, such that $g_u(u,[u+u',\mathfrak{m}])=0$. Then we have
\begin{eqnarray*}
Q(\mathfrak{m},[u+u',A_u(u)])=Q([\mathfrak{m}, u+u'],A_u(u))=
 g_u([\mathfrak{m},u+u']_{\mathfrak{m}},u)=0,
\end{eqnarray*}
i.e., $[u+u',A_u(u)]\in\mathfrak{h}$.

On the other hand, (\ref{004}) implies $g_u([u,\mathfrak{h}],u)=0$. So we have
$$Q([u,A_u(u)],\mathfrak{h})=Q(A_u(u),[u,\mathfrak{h}])=g_u([u,\mathfrak{h}],u)=0,$$
i.e., $[u,A_u(u)]\in\mathfrak{m}$. Together with the obvious fact that $[u',A_u(u)]\in[\mathfrak{h},\mathfrak{m}]\subset\mathfrak{m}$, we get
$[u+u',A_u(u)]\in\mathfrak{m}$.

To summarize, above argument proves $[u+u',A_u(u)]=0$, i.e., it proves one direction of Lemma \ref{lemma3.4}. The  other direction can be proved similarly.
\end{proof}

For the natural reductiveness, we have
\begin{lemma}\label{lemma 3.6}
$(G/H,F)$ is naturally reductive if and only if
$
   [u,A_u(u)]=0.
  $
for any nonzero vector $u\in \mathfrak{m}$.
\end{lemma}

We skip its proof, which is similar to that of Lemma \ref{lemma3.4}.

\subsection{Standard homogeneous $(\alpha_1,\alpha_2,\cdots, \alpha_s)$-metric}
\label{section-3-3}

Now we generalize the standard homogeneous $(\alpha_1,\alpha_2)$-metric in \cite{ZX}.
Suppose that $\mathfrak{m}$ has an $\mathrm{Ad}(H)$-invariant $Q$-orthogonal decomposition
$\mathfrak{m}=\mathfrak{m}_1+\cdots+\mathfrak{m}_s$. We have quadratic functions $\alpha_i(y)=\alpha_i(y_i)=Q(y_i,y_i)$, with $y=y_1+\cdots+y_s$ and $y_i\in\mathfrak{m}_i$ for each $i$. Since each $\alpha_i$ is $\mathrm{Ad}(H)$-invariant, a Minkowski norm on $\mathfrak{m}$, which is of the form
$F=\sqrt{L(\alpha_1^2,\cdots,\alpha_s^2)})$, is also $\mathrm{Ad}(H)$-invariant. It induces a homogeneous Finsler metric $F$ on $G/H$, which is called a {\it standard homogeneous $(\alpha_1,\cdots,\alpha_s)$-metric} (or simply a {\it standard homogeneous Finsler metric}).
The function $L=L(\theta_1,\cdots,\theta_s)$ here
is a positive smooth function on $[0,+\infty)^s\backslash\{(0,\cdots,0)\}$.
Obviously it must satisfy the positive 1-homogeneity:
\begin{equation*}
  L(t\theta_1,t\theta_2,\cdots,t\theta_s)=tL(\theta_1,\theta_2,\cdots,\theta_s), \quad\forall t\geq0,\theta_1\geq0,\cdots\theta_s\geq0.
\end{equation*}
Further more, $L$ must satisfy certain differential inequalities for $F=\sqrt{L(\alpha_1^2,\cdots,\alpha_s^2)}$ to be strong convex (see \cite{DX2016,TX2023}
when $s=2$).

Using the form $F=\sqrt{L(\alpha_1^2,\cdots,\alpha_s^2)}$, we may more generally define non-standard homogeneous $(\alpha_1,\cdots,\alpha_s)$-metrics and non-homogeneous $(\alpha_1,\cdots,\alpha_s)$-metrics.

\begin{remark}
Any standard homogeneous Finsler metric is reversible because the Minkowski norm has
a canonical block-diagonal linear $O(\mathfrak{m}_1)\times\cdots\times O(\mathfrak{m}_s)$-symmetry, where each orthogonal group $O(\mathfrak{m}_i)$ is
with respect to $Q|_{\mathfrak{m}_i\times\mathfrak{m}_i}$. This linear symmetry characterizes $(\alpha_1,\cdots,\alpha_s)$-metrics. If we replace it by the weaker $SO(\mathfrak{m}_1)\times\cdots\times SO(\mathfrak{m}_s)$-symmetry, then irreversible
standard homogeneous Finsler metrics may be defined $\mathrm{(}$for example, standard homogeneous $(\alpha,\beta)$-metric, etc.$\mathrm{)}$.
\end{remark}

\subsection{Metric operator of a standard homogeneous $(\alpha_1,\cdots,\alpha_s)$-metric}
Assume that $F=$ $\sqrt{L(\alpha_1^2,\cdots,\alpha_s^2)}$ is a standard homogeneous $(\alpha_1,\alpha_2,\cdots,\alpha_s)$-metric on $G/H$ with respect to the decomposition
$\mathfrak{m}=\mathfrak{m}_1+\cdots+\mathfrak{m}_s$.
Denote by $\{e^i_{l},\forall  1\leq l\leq n_i\}$ be a $Q$-orthonormal basis in $\mathfrak{m}_i$, where $n_i=\dim\mathfrak{m}_i$.
Any $y\in \mathfrak{m}$ can be presented as
$y=\sum_{i=1}^s y_i$ with $y_i=\sum_{l=1}^{n_i}y_i^le^i_l\in\mathfrak{m}_i$. When $y\neq0$, the fundamental tensors at $y$ are
\begin{eqnarray*}
& &g_y(e_i^a,e_j^a)
=2\frac{\partial^2 L}{\partial\theta_a^2}y^i_a y^j_a+\frac{\partial L}{\partial\theta_a}\delta_{ij},
\quad\forall 1\leq a\leq s, 1\leq i,j\leq n_a,\\
& &g_y(e_i^a,e_j^b)=2\frac{\partial^2 L}{\partial \theta_a\partial \theta_b}y^i_a y^j_b,
\quad\forall 1\leq a\neq b\leq s, 1\leq i\leq n_a, 1\leq j\leq n_b,
\end{eqnarray*}
where the partial derivatives of $L$ are evaluated at $(\alpha_1^2(y_1),\cdots,\alpha_s^2(y_s))$
(same below).
Then  the Hessian matrix of $\tfrac{1}{2}F^2$
can be presented as $(G_{n_i\times n_j})_{1\leq i,j\leq s}$,

where
\begin{equation*}
G_{n_a\times n_a}(y)=\frac{\partial L}{\partial\theta_a}
\begin{pmatrix}
                  1       &                     &        &                    \\
                          & 1                   &        &                    \\
                          &                     & \ddots &                    \\
                          &                     &        & 1
\end{pmatrix}
+2\frac{\partial^2L}{\partial \theta_a^2}
\begin{pmatrix}
     y_a^1y_a^1           & y_a^1y_a^2          & \cdots &y_a^1y_a^{n_a} \\
     y_a^2y_a^1           & y_a^2y_a^2          & \cdots &y_a^2y_a^{n_a} \\
        \vdots            &    \vdots           & \ddots &   \vdots           \\
     y_a^{n_a}y_a^1   & y_a^{n_a}y_a^{n_2}      & \cdots &y_a^{n_a}y_a^{n_a}
\end{pmatrix}
\end{equation*}
for $1\leq a\leq s$, and
\begin{equation*}
G_{n_a\times n_b}(y)=
2\frac{\partial^2L}{\partial \theta_a\partial \theta_b}
\begin{pmatrix}
     y_a^1y_b^1           & y_a^1y_b^2          & \cdots &y_a^1y_b^{n_b} \\
     y_a^2y_b^1           & y_a^2y_b^2          & \cdots &y_a^2y_b^{n_b} \\
        \vdots            &    \vdots           & \ddots &   \vdots           \\
     y_a^{n_a}y_b^1   & y_a^{n_a}y_b^{n_2}      & \cdots &y_a^{n_a}y_b^{n_b}
\end{pmatrix}
\end{equation*}
for $1\leq a\neq b\leq s$.

Now we discuss the metric operator $A_u$
for $(G/H,F)$ with $u\in\mathfrak{m}\backslash\{0\}$.
The previous calculation for the fundamental tensor implies

\begin{lemma}\label{proposition 3.10}
For any nonzero vector $u=u_1+u_2+\cdots+u_s$ with $u_i\in\mathfrak{m}_i$ for each $i$, $v,v'\in \mathfrak{m}_j$, $w\in \mathfrak{m}_k$ and $1\leq j\neq k\leq s$, we have
\begin{equation*}
  Q(A_u(v),w)=2\frac{\partial^2L}{\partial \theta_j\partial \theta_k}Q(u_j,v)Q(u_k,w)
\end{equation*}and
\begin{equation*}
  Q(A_u(v),v')=\frac{\partial L}{\partial \theta_j}Q(v,v')+2\frac{\partial^2L}{\partial \theta_j^2}Q(u_j,v)Q(u_j,v').
\end{equation*}
\end{lemma}

Let $I=\{i_1,\cdots,i_k\}$ be any subset of $\{1,\cdots,s\}$
and $\mathfrak{m}_I=\mathfrak{m}_{i_1}+\cdots+\mathfrak{m}_{i_k}$.
Obviously the $Q$-orthogonal complement of $\mathfrak{m}_I$ is $\mathfrak{m}_J$
for $J=\{1,\cdots,s\}\backslash I$.

\begin{lemma}\label{Proposition 3.8}
For each nonzero vector $u=u_{1}+\cdots+u_{s}$ with $u_i\in \mathfrak{m}_i$ for each $i$,
we have $A_u(u)=\sum_{i=1}^s\frac{\partial L}{\partial \theta_i}u_i$.
In particular, we have $A_u(u)\in\mathfrak{m}_I$ for $u\in\mathfrak{m}_I\backslash\{0\}$.
\end{lemma}

\begin{proof} Let $u_i=\sum_{l=1}^{n_i}u_i^l e^i_l$.
To prove the first statement in Lemma \ref{Proposition 3.8},
we only need to verify $Q(A_u(u),e^j_k)= \tfrac{\partial L}{\partial \theta_j} u_j^k$ for any $1\leq j\leq s$ and $1\leq k\leq n_i$.
By Lemma \ref{proposition 3.10},
\begin{eqnarray}\label{005}
Q(A_u(u),e^j_k)&=&\sum_{i=1}^s\sum_{l=1}^{n_i}
u_i^l Q(A_u(e^i_l), e^j_k) =\tfrac{\partial L}{\partial \theta_j} u_j^k+
\sum_{i=1}^s\sum_{l=1}^{n_i}2\tfrac{\partial^2 L}{\partial\theta_i\partial\theta_j} (u^i_l)^2 u_j^k\nonumber\\
&=& \tfrac{\partial L}{\partial \theta_j} u_j^k+ 2u_j^k\sum_{i=1}^s
\tfrac{\partial^2 L}{\partial\theta_i\partial\theta_j}\alpha_i^2(u_i).
\end{eqnarray}
Since the partial derivatives of $L$ in (\ref{005}) are evaluated at $(\alpha_1^2(u_1),\cdots,\alpha_s^2(u_s))$ and $\tfrac{\partial L}{\partial\theta_k}$ is positive 0-homogeneous, $\sum_{i=1}^s
\tfrac{\partial^2 L}{\partial\theta_i\partial\theta_j}\alpha_i^2(u_i)=0$.
The first statement in Lemma \ref{Proposition 3.8} is proved. The second statement follows immediately.
\end{proof}
\section{Compact homogeneous spaces with two isotropy summands}
\subsection{Some easy observations and useful results}
Throughout this section, we assume that $G/H$ is a connected simply connected homogeneous manifold with a compact connected simple $G$ (then $H$ must be connected as well), and an $\mathrm{Ad}(H)$-invariant $Q$-orthogonal decomposition $\mathfrak{m}=\mathfrak{m}_1+\mathfrak{m}_2$, in which each $\mathfrak{m}_i$ is $H$-irreducible. The classification of all these $G/H$ (in the Lie algebraic level) is given in \cite{DK} (see also Table 1-2 in \cite{CN}), which can
be sorted into two subclasses, either $H$ is maximal in $G$, or we can find a compact connected subgroup $K$ satisfying $H\subsetneq K\subsetneq G$.

Fortunately, we will not essentially use this classification. Instead,
the following two theorems in \cite{CN} are crucial for proving our Theorem \ref{theorem 1.1}.

\begin{theorem}\label{thm-4}
Let $G/H$ be a connected simply connected homogeneous manifold which isotropy representation is the sum of two irreducible summands, and $G$ is a compact connected simple Lie group.
Then $G/H$ admits a
$G$-invariant g.o. Riemannian metric which is not normal  homogeneous, if and only if there exists a compact connected Lie subgroup $K$ such that $H\subset K\subset G$, $G/K$ is symmetric, and
 $(H,K,G)$ belongs to the list in Theorem \ref{theorem 1.1} in the Lie algebraic level.
\end{theorem}

\begin{theorem}\label{thm-5}
Keeping all assumptions and notations in this section.
The following statements are equivalent:
\begin{enumerate}
\item the metric $\lambda Q|_{\mathfrak{m}_1\times\mathfrak{m}_1}+\mu Q|_{\mathfrak{m}_2\times\mathfrak{m}_2}$ is g.o. for some $\lambda\neq\mu$;
\item the metric $\lambda Q|_{\mathfrak{m}_1\times\mathfrak{m}_1}+\mu Q|_{\mathfrak{m}_2\times\mathfrak{m}_2}$ is g.o. for any $\lambda$, $\mu$;
\item for any $X\in\mathfrak{m}_1$ and $Y\in\mathfrak{m}_2$, there is a unique $Z_X\in\widetilde{C}_\mathfrak{h}(X+Y)\cap C_\mathfrak{h}(X)$ and a unique
    $Z_Y\in\widetilde{C}_\mathfrak{h}(X+Y)\cap C_\mathfrak{h}(Y)$ such that
    $[X,Y]=[Z_Y,X]+[Z_X,Y]$.
\end{enumerate}
\end{theorem}

For any $U\in\mathfrak{m}$, $C_\mathfrak{h}(U)=\{u| u\in\mathfrak{h},[u,X]=0\}$ is
the centralizer of $U\in\mathfrak{m}$ in $\mathfrak{h}$.
By $\widetilde{C}_\mathfrak{h}(U)$, we denote the $Q$-orthogonal complement of $C_\mathfrak{h}(U)$
in the normalizer $N_\mathfrak{h}(C_\mathfrak{h}(U))=\{u| u\in\mathfrak{h},[u, C_\mathfrak{h}(U)]\subset C_\mathfrak{h}(U)\}$.

We will also use Lemma 4 in \cite{CN} (see also Lemma 7 in \cite{Niko}), i.e.,

\begin{lemma}\label{lemma 4.1}
In the above assumptions, we have $[\mathfrak{m}_1, \mathfrak{m}_2]\neq 0$.
\end{lemma}

\subsection{Descriptions for g.o. and naturally reductive properties}
Let $G/H$ be the homogeneous manifold described in the previous subsection, and $F=\sqrt{L(\alpha_1^2,\alpha_2^2)}$ a standard homogeneous $(\alpha_1,\alpha_2)$-metric on $G/H$ with respect to the given decomposition $\mathfrak{m}=\mathfrak{m}_1+\mathfrak{m}_2$.

\begin{lemma}\label{proposition 5.2}
$F=\sqrt{(\alpha_1^2,\alpha_2^2)}$ is g.o. on $G/H$ if and only if, for any $u=u_1+u_2$ with $u_1\in \mathfrak{m}_1\backslash\{0\}$ and $u_2\in \mathfrak{m}_2\backslash\{0\}$,
there is $u'\in \mathfrak{h}$ satisfying
\begin{equation}\label{equation 5.5}
  (\frac{\partial L}{\partial \theta_1}-\frac{\partial L}{\partial \theta_2})[u_1,u_2]=
  \frac{\partial L}{\partial \theta_1}[u',u_1]+ \frac{\partial L}{\partial \theta_2}[u',u_2].
\end{equation}
\end{lemma}

\begin{proof}

By Lemma \ref{lemma3.4}, $F$ is g.o. on $G/H$ if and only if for any nonzero vector $u=u_1+u_2$ with
$u_1\in \mathfrak{m}_1$ and $u_2\in \mathfrak{m}_2$, there exists $u'\in\mathfrak{h}$ such that
$[u'+u,A_u(u)]=0$. By Lemma \ref{Proposition 3.8},
\begin{eqnarray*}
[u'+u,A_u(u)]&=&[u'+u_1+u_2,\tfrac{\partial L}{\partial\theta_1}u_1+\tfrac{\partial L}{\partial \theta_2}u_2]\\
&=&\tfrac{\partial L}{\partial\theta_1}[u',u_1]+\tfrac{\partial L}{\partial\theta_2}[u',u_2]
+(\tfrac{\partial L}{\partial \theta_2}-\tfrac{\partial L}{\partial \theta_1})[u_1,u_2].
\end{eqnarray*}
So $[u'+u,A_u(u)]=0$ is equivalent to (\ref{equation 5.5}). Finally, we notice that (\ref{equation 5.5})
is always valid when $u\in\mathfrak{m}_1\cup\mathfrak{m}_2$, because we can choose $u'=0$ in this case. The
proof of Lemma \ref{proposition 5.2} is finished.
\end{proof}

As in \cite{ZX}, we alternatively present $F=\sqrt{L(\alpha_1^2,\alpha_2^2)}$ as $F=\alpha\varphi(\theta)$, Where $\alpha^2(u)=\alpha_1^2(u_1)+\alpha_2^2(u_2)=Q(u,u)$, $L(1-\theta^2,\theta^2)=\varphi^2(\theta)$, and $\theta=\alpha_2(u_2)/\alpha(u)$. Then, we have
\begin{equation*}
  L(\theta_1,\theta_2)=(\theta_1+\theta_2)\varphi^2(\theta),\
  \theta=\sqrt{\frac{\theta_2}{\theta_1+\theta_2}},
\end{equation*}
and
\begin{equation}\label{equation 5.6}
  \frac{\partial L}{\partial \theta_1}=\varphi^2(\theta)-\theta\varphi(\theta)\varphi'(\theta),\ \
  \frac{\partial L}{\partial \theta_2}=\varphi^2(\theta)-(\theta-\frac{1}{\theta})\varphi(\theta)\varphi'(\theta).
\end{equation}
Hence Lemma \ref{proposition 5.2} can be translated to
\begin{lemma}\label{lemma-7}
Let $(G/H,F)$ be a standard homogeneous $(\alpha_1, \alpha_2)$-space with respect to the decomposition $\mathfrak{m}=\mathfrak{m}_1+\mathfrak{m}_2$.
Then $F=\alpha\varphi(\alpha_2/\alpha)$ is $G$-g.o. if and only if for any $u_1\in\mathfrak{m}_1\backslash\{0\}$ and $u_2\in\mathfrak{m}_2\backslash\{0\}$,
there exists $u'\in\mathfrak{h}$, such that
\begin{equation}\label{equation 5.7}
  -\frac{\varphi'(\theta)}{\theta}[u_1,u_2]=
  (\varphi(\theta)-\theta\varphi'(\theta))[u',u_1]+
  (\varphi(\theta)-(\theta-\frac{1}{\theta})\varphi'(\theta))[u',u_2],
\end{equation}
where $\theta=\alpha_2(u_2)/\alpha(u)\in(0,1)$.
\end{lemma}

\begin{remark} \label{remark-1}
Lemma $\mathrm{2.4}$ in \cite{ZX} indicates that
the coefficients $\varphi(\theta)-\theta\varphi'(\theta)$ and $\varphi(\theta)-(\theta-\tfrac{1}{\theta})\varphi'(\theta)$ in the right side of $\mathrm{(\ref{equation 5.7})}$
are always positive.
\end{remark}

When $u'=0$ is always taken, Lemma \ref{lemma 3.6} and similar calculation provide

\begin{lemma}\label{lemma-8}
$F=\alpha\varphi(\alpha_2/\alpha)$ is naturally reductive on $G/H$ if and only if
it is a normal homogeneous Riemannian metric.
\end{lemma}

\begin{proof}
Assume that $F$ is naturally reductive on $G/H$. Lemma \ref{lemma 4.1} provides $v_1\in\mathfrak{m}_1$
and $v_2\in\mathfrak{m}_2$ such that $[v_1,v_2]\neq0$. Take $u_1=v_1$ and $u_2=\lambda v_2$
with $\lambda>0$, then $\theta=\alpha_2(u_2)/\alpha(u)$ can exhaust all numbers in $(0,1)$.
By Lemma \ref{lemma 3.6}, Lemma \ref{Proposition 3.8}, and similar calculations as for Lemma \ref{lemma-7}, we get $-\frac{\varphi'(\theta)}{\theta}[u_1,u_2]=0$. Since $[u_1,u_2]=\lambda[v_1,v_2]\neq0$, $\varphi(\theta)$ must be a constant function, i.e., $F$ is a
normal homogeneous Riemannian metric.

To summarize, above argument proves one side of Lemma \ref{lemma-8}. The other side is obvious.
\end{proof}

\subsection{Proof of Theorem \ref{theorem 1.1}}

\begin{lemma}\label{proposition 5.7}
Keeping all assumptions and notations for $G/H$,
then the following conditions are equivalent:
\begin{enumerate}
\item[(1)] any standard homogeneous $(\alpha_1,\alpha_2)$-metric $F=\alpha\varphi(\theta)$ on $G/H$, with respect to the given decomposition $\mathfrak{m}=\mathfrak{m}_1+\mathfrak{m}_2$, is g.o.;
  \item[(2)] there exists a standard homogeneous $(\alpha_1,\alpha_2)$-metric $F=\alpha\varphi(\theta)$ on $G/H$, with respect to the given decomposition $\mathfrak{m}=\mathfrak{m}_1+\mathfrak{m}_2$, which is g.o. but not naturally reductive;
  \item[(3)] for any $u_1\in \mathfrak{m}_1$ and $u_2\in \mathfrak{m}_2$, there exist
  $Z_{u_1}\in \widetilde{C}_{\mathfrak{h}}(u_1+u_2)\cap C_{\mathfrak{h}}(u_1)$ and  $Z_{u_2}\in \widetilde{C}_{\mathfrak{h}}(u_1+u_2)\cap C_{\mathfrak{h}}(u_2)$ such that
  $[u_1,u_2]=[Z_{u_2},u_1]+[Z_{u_1},u_2]$.
   \end{enumerate}
Moreover, $Z_{u_1}$ and $Z_{u_2}$ in $\mathrm{(3)}$ are unique.
\end{lemma}

\begin{proof}First, we prove the statement from (2) to (3). Choose any $u_1\in\mathfrak{m}_1$ and
$u_2\in\mathfrak{m}_2$. If one of them vanishes, $Z_{u_1}=0$ and $Z_{u_2}=0$ satisfy the requirement in (3). So we may further assume $u_1\neq0$ and $u_2\neq0$. Since $F$ is not naturally reductive, i.e., $\varphi$ is not constant on $(0,1)$, we can find $\lambda>0$, such that $\theta=\alpha_2(u_2)/\alpha(\lambda u_1+u_2)\in(0,1)$
satisfies $\varphi'(\theta)\neq0$. Applying Lemma \ref{lemma-7} to $\lambda u_1$ and $u_2$,we get
$u'\in\mathfrak{h}$ satisfying
$[u_1,u_2]=a[u',u_1]+b[u',u_2]$, in which
\begin{equation}\label{006}
a=\tfrac{\theta^2\varphi'(\theta)-\theta\varphi(\theta)}{\varphi'(\theta)}\quad\mbox{and}\quad
b=\tfrac{(\theta^2-1)\varphi'(\theta)-\theta\varphi(\theta)}{\lambda\varphi'(\theta)}.
\end{equation}
By Remark \ref{remark-1}, both $a$ and $b$ are nonzero.

\noindent{\bf Claim A}: we can choose $u'$ from $\widetilde{C}_\mathfrak{h}(u_1+u_2)$.

To prove Claim A, we first verify $u'\in N_{\mathfrak{h}}( C_\mathfrak{h}(u_1+u_2))$.
For any $v\in C_\mathfrak{h}(u_1+u_2)=C_\mathfrak{h}(u_1)\cap C_\mathfrak{h}(u_2)$, we have
\begin{eqnarray*}
& &a[[v,u'],u_1]+b[[v,u'],u_2]=[v,a[u',u_1]+b[u',u_2]]\\
&=&[v,[u_1,u_2]]=[[v, u_1],u_2]+[u_1,[v,u_2]]=0,
\end{eqnarray*}
in which the two summands in the left side belong to $\mathfrak{m}_1$ and $\mathfrak{m}_2$ respectively. So we get $[[v,u'],u_1]=[[v,u'],u_2]=0$, i.e., $[v,u']\in C_\mathfrak{h}(u_1)\cap C_\mathfrak{h}(u_2)=C_\mathfrak{h}(u_1+u_2)$. So $u'\in N_{\mathfrak{h}}(C_\mathfrak{h}(u_1+u_2))$.

Notice that the vector $u'$ provided by Lemma \ref{lemma-7} for $\lambda u_1$ and $u_2$ is not unique. It can be replaced by any vector in $u'+C_\mathfrak{h}(u_1+u_2)$. In particular, we can choose it from the $Q$-orthogonal complement $\widetilde{C}_\mathfrak{h}(u_1+u_2)$ of
$C_\mathfrak{h}(u_1+u_2)$ in $N_\mathfrak{h}(u_1+u_2)$, which proves Claim A.

We can also apply Lemma \ref{lemma-7} to $\lambda u_1$ and $-u_2$, and then get $u''\in \widetilde{C}_\mathfrak{h}(u_1-u_2)=\widetilde{C}_\mathfrak{h}(u_1+u_2)$, such that
$-[u_1,u_2]=a[u'',u_1]-b[u'',u_2]$. Notice that $a$ and $b$ here are also given by (\ref{006}),
because $\alpha_2(u_2)=\alpha_2(-u)$ and $\alpha(\lambda u_1+u_2)=\alpha(\lambda u_1-u_2)$.

Adding the two equalities containing $u'$ and $u''$ respectively, we get $a[u'+u'',u_1]+b[u'-u'',u_2]=0$, in which the two
summand in the left side belong to $\mathfrak{m}_1$ and $\mathfrak{m}_2$ respectively. So $u'+u''\in C_\mathfrak{h}(u_1)$ and $u'-u''\in C_\mathfrak{h}(u_2)$. Define $Z_{u_1}=\tfrac{b}{2}(u'+u'')\in\widetilde{C}_\mathfrak{h}(u_1+u_2)\cap C_\mathfrak{h}(u_1)$ and $Z_{u_2}=\tfrac{a}{2}(u'-u'')\in\widetilde{C}_\mathfrak{h}(u_1+u_2)\cap C_\mathfrak{h}(u_2)$.
Then $u'=\tfrac{1}{b}Z_{u_1}+\tfrac{1}{a}Z_{u_2}$ and
\begin{eqnarray*}
[u_1,u_2]&=&a[u',u_1]+b[u',u_2]\\
&=&[\tfrac{a}{b}Z_{u_1},u_1]+[Z_{u_2},u_1]
+[Z_{u_1},u_2]+[\tfrac{b}{a}Z_{u_2},u_2]\\
&=&[Z_{u_2},u_1]
+[Z_{u_1},u_2].
\end{eqnarray*}
The proof of the statement from (1) to (2) is finished.

Next, we prove the statement from (3) to (1). Choose any $u_1\in\mathfrak{m}_1\backslash\{0\}$ and $u_2\in\mathfrak{m}_2\backslash\{0\}$, then we have $Z_{u_1}$ and $Z_{u_2}$ provided by (3). If $\varphi'(\theta)\neq0$ for $\theta=\tfrac{\alpha_2(u_2)}{\alpha(u)}\in(0,1)$, we choose $u'=\tfrac{1}{b}Z_{u_1}+\tfrac{1}{a}Z_{u_2}$, in which $a$ and $b$ are given in (\ref{006}). Otherwise, we choose
$u'=0$. In both cases, (\ref{equation 5.7}) is satisfies, and then Lemma \ref{lemma-7} provides the g.o. property. The proof of the statement from (3) to (1) is finished. The proof of the statement from (1)
to (2) is obvious.

Finally, we prove the uniqueness of $Z_{u_1}$ and $Z_{u_2}$ in (3).
Assume $Z'_{u_1}\in\widetilde{C}_\mathfrak{h}(u_1+u_2)\cap C_\mathfrak{h}(u_1)$ and $Z'_{u_2}\in\widetilde{C}_\mathfrak{h}(u_1+u_2)\cap C_\mathfrak{h}(u_2)$ satisfy $[u_1,u_2]=[Z'_{u_2},u_1]+[Z'_{u_1},u_2]$. Then we have $[Z_{u_2}-Z'_{u_2},u_1]+[Z_{u_1}-Z'_{u_1},u_2]=0$. By earlier argument,
$Z_{u_2}-Z'_{u_2}\in C_\mathfrak{h}(u_1)$, i.e., $Z_{u_2}-Z'_{u_2}\in C_\mathfrak{h}(u_1)\cap
C_\mathfrak{h}(u_2)=C_\mathfrak{h}(u_1+u_2)$. On the other hand,
$Z_{u_2}-Z'_{u_2}\in\widetilde{C}_\mathfrak{h}(u_1+u_2)$. So $Z_{u_2}-Z'_{u_2}=0$. For the same reason, $Z_{u_1}-Z'_{u_1}=0$, which ends the proof.
\end{proof}

\begin{proof}[Proof of Theorem \ref{theorem 1.1}]
First, we prove Theorem \ref{theorem 1.1} with the additional assumption $(\mathfrak{g},\mathfrak{h})\neq(D_4,G_2)$. Then the two summands in the isotropy action for $G/H$ are inequivalent $H$-representations. So the given decomposition $\mathfrak{m}=\mathfrak{m}_1+\mathfrak{m}_2$ is the only nontrivial $\mathrm{Ad}(H)$-invariant decomposition for $\mathfrak{m}$. In this case, any homogeneous Riemannian metric on $G/H$ is a standard homogeneous $(\alpha_1,\alpha_2)$-metric, and any standard homogeneous $(\alpha_1,\alpha_2)$-metric on $G/H$ must be with respect to $\mathfrak{m}=\mathfrak{m}_1+\mathfrak{m}_2$.

Assume that $G/H$ admits a standard homogeneous $(\alpha_1,\alpha_2)$-metric which is g.o. but not naturally reductive. Then
Lemma \ref{proposition 5.7} indicates that it satisfies (2) in Theorem \ref{thm-5}. By Theorem \ref{thm-5}, $G/H$ admits a homogeneous Riemannian metric which is g.o. but not normal. Theorem \ref{thm-4} tells us that $G/H$ belongs to the list in Theorem \ref{theorem 1.1}, in the Lie algebraic level.
Conversely, assume that $G/H$ belongs to the list (except the first one) in Theorem \ref{theorem 1.1}.
It admits a homogeneous Riemannian metric which is g.o. but not normal. Then Theorem \ref{thm-5} indicates that $G/H$ satisfies the requirement in (3) in Lemma \ref{proposition 5.7}, and Lemma \ref{proposition 5.7} tells us that any standard homogeneous $(\alpha_1,\alpha_2)$-metric $F=\alpha\varphi(\alpha_2/\alpha)$ on $G/H$ is g.o.. We may choose a non-constant function $\varphi(\theta)$ here. Then by Lemma \ref{lemma-8}, we see that this $F$ is not naturally reductive.
 The proof of Theorem \ref{theorem 1.1} when $(\mathfrak{g},\mathfrak{h})\neq(D_4,G_2)$ is finished.

 Next, we consider the case $(\mathfrak{g},\mathfrak{h})=(D_4,G_2)$. In this case, $G/H$ coincides
 with $Spin(8)/G_2$ in the Lie algebraic level, so any homogeneous Finsler metric is locally isometric to a homogeneous Finsler metric on $Spin(8)/G_2$. Because $Spin(8)/G_2$ is weakly symmetric \cite{DH},
 any homogeneous Finsler metric on $Spin(8)/G_2$ is g.o. \cite{D1}. So any homogeneous Finsler metric on $G/H$ is g.o. as well.
This ends the proof of Theorem \ref{theorem 1.1}.
\end{proof}

\section{Standard homogeneous $(\alpha_1,\alpha_2,\alpha_3)$-metrics on generalized Wallach spaces  }

\subsection{Classification of generalized Wallach spaces}
Throughout this section, we assume that $G/H$ is a simply connected
homogeneous manifold, such that the compact connected semi simple
$G$ acts effectively on $G/H$, and we have an $Q$-orthogonal decomposition
$\mathfrak{m}=\mathfrak{m}_1+\mathfrak{m}_2+\mathfrak{m}_3$
(we call it the {\it canonical decomposition}),
where each $\mathfrak{m}_i$ is an $\mathrm{Ad}(H)$-invariant nonzero space on which the $H$-action is irreducible,
and $[\mathfrak{m}_i,\mathfrak{m}_i]\subseteq \mathfrak{h}$ for each $i$.
This $G/H$ is called a {\it generalized Wallach spaces}. Theorem 1 in \cite{Niko1} classifies generalized Wallach spaces, which consist of the following three subclasses:
\begin{enumerate}
  \item[{\bf Type I}:] if $[\mathfrak{m}_i,\mathfrak{m}_j]=0$ when $\{i,j,k\}=\{1,2,3\}$,
  $G/H$ is the product of three irreducible symmetric spaces of compact type;
  \item[{\bf Type II}:] if $[\mathfrak{m}_i,\mathfrak{m}_j]=\mathfrak{m}_k$ when $\{i,j,k\}=\{1,2,3\}$ and $G$ is simple,  the pair $(\mathfrak{g}, \mathfrak{h})$ is one of the pairs in Table \ref{Table B};
  \item[{\bf Type III}:] if $[\mathfrak{m}_i,\mathfrak{m}_j]=\mathfrak{m}_k$ when $\{i,j,k\}=\{1,2,3\}$ and $G$ is not simple, $G=K\times K\times K$ and $H=\mathrm{diag}(K)\subseteq G$, where $K$ is a connected simply connected compact simple Lie group, $\mathfrak{m}_1=\{(X,X,-X,-X),\forall X\in\mathfrak{k}\}$,
      $\mathfrak{m}_2=\{(X,-X,X,-X),\forall X\in\mathfrak{k}\}$,
      $\mathfrak{m}_3=\{(X,-X,-X,X),\forall X\in\mathfrak{k}\}$ and
      $\mathfrak{k}=\mathrm{Lie}(K)$. This $G/H$ is called a {\it Ledger-Obata space} in some literatures.
\end{enumerate}
\begin{table}[htbp]
\newcommand{\tabincell}[2]{\begin{tabular}{@{}#1@{}}#2\end{tabular}}
\centering
\caption{$(\mathfrak{g},\mathfrak{h})$ for the generalized Wallach space $G/H$
with a simple $G$}\label{Table B}
\begin{tabular}{|c|c|c|c|}
  \hline
   $\mathfrak{g}$         & $\mathfrak{h}$    &$\mathfrak{g}$           &$\mathfrak{h}$          \\
  \hline
  $so(k+l+m)$ & $so(k)\oplus so(l)\oplus so(m)$ &$e_7$ & $so(8)\oplus3sp(1)$  \\
  \hline
  $su(k+l+m)$ & $s(u(k)\oplus u(l)\oplus u(m))$ &$e_7$ &  $su(6)\oplus sp(1)\oplus R$ \\
  \hline
  $sp(k+l+m)$ & $Sp(k)\oplus sp(l)\oplus sp(m)$ &$e_7$ & $so(8)$    \\
  \hline
  $su(2l), l\geq 2$ & $u(l)$   &$e_8$  & $so(12)\oplus 2sp(1)$     \\
  \hline
  $so(2l), l\geq 4$ & $u(1)\oplus u(l-1)$ &$e_8$  &$so(8)\oplus so(8)$   \\
  \hline
  $e_6$ & $su(4)\oplus 2sp(1)\oplus R$   &$f_4$  &$so(5)\oplus 2sp(1)$  \\
  \hline
  $e_6$ & $so(8)\oplus R^2$  &$f_4$  &$so(8)$  \\
  \hline
  $e_6$ & $sp(3)\oplus sp(1)$  &    &   \\
  \hline
\end{tabular}
\end{table}

Let $F=\sqrt{L(\alpha_1^2,\alpha_2^2,\alpha_3^3)}$ be a standard homogeneous $(\alpha_1,\alpha_2,\alpha_3)$-metric on $G/H$ with respect to the canonical
decomposition.
We will discuss the g.o. property of $(G/H,F)$.
Type I is easy, because when $G/H$ is of Type I, we have $[\mathfrak{m},\mathfrak{m}]\subset\mathfrak{h}$, and then any homogeneous Finsler metric $F$ on $G/H$ is naturally reductive.

To discuss Type II and Type III, we need the following criterion.

\begin{lemma}\label{lemma-9}
Keeping all assumptions and notations in this section, then the homogeneous metric $F=\sqrt{L(\alpha_1^2,\alpha_2^2,\alpha_3^2)}$ is g.o. on $G/H$ if and only if for any nonzero vector $u=u_1+u_2+u_3$ with $u_i\in \mathfrak{m}_i$ for each $i$, there exists $u'\in \mathfrak{h}$ satisfying
\begin{equation}\label{equation 6.9}
\begin{cases}
\frac{\partial L}{\partial \theta_1}[u',u_1]+
(\frac{\partial L}{\partial \theta_3}-\frac{\partial L}{\partial \theta_2})[u_2,u_3]=0, \\
\\
\frac{\partial L}{\partial \theta_2}[u',u_2]+
(\frac{\partial L}{\partial \theta_3}-\frac{\partial L}{\partial \theta_1})[u_1,u_3]=0, \\
\\
\frac{\partial L}{\partial \theta_3}[u',u_3]+
(\frac{\partial L}{\partial \theta_2}-\frac{\partial L}{\partial \theta_1})[u_1,u_2]=0.
\end{cases}
\end{equation}
\end{lemma}

\begin{proof} Assume that $F$ is g.o. on $G/H$. Choose any nonzero vector $u=u_1+u_2+u_3$ with $u_i\in\mathfrak{m}_i$ for each $i$. By Lemma \ref{Proposition 3.8}, $A_u(u)=\sum_{i=1}^3\tfrac{\partial L}{\partial\theta_i}u_i$. Lemma \ref{lemma3.4} provides $u'\in\mathfrak{h}$, which satisfies $[u+u',A_u(u)]=0$, i.e.,
\begin{eqnarray}
& &(\tfrac{\partial L}{\partial \theta_1}[u',u_1]+(\tfrac{\partial L}{\partial\theta_3}-
\tfrac{\partial L}{\partial\theta_2})[u_2,u_3])
+(\frac{\partial L}{\partial \theta_2}[u',u_2]+
(\frac{\partial L}{\partial \theta_3}-\frac{\partial L}{\partial \theta_1})[u_1,u_3])
\nonumber\\
& &+(\frac{\partial L}{\partial \theta_3}[u',u_3]+
(\frac{\partial L}{\partial \theta_2}-\frac{\partial L}{\partial \theta_1})[u_1,u_2])=0.\label{007}
\end{eqnarray}
Notice that the three summands in (\ref{007}) belong to the three distinct $\mathfrak{m}_i$ respectively, we get (\ref{equation 6.9}) in one direction. Reversing above discussion, the other direction can also be proved.
\end{proof}

\subsection{Discussion for Type II and Type III}

When $G/H$ is of Type II or Type III,
$[\mathfrak{m}_i,\mathfrak{m}_j]=\mathfrak{m}_k$ for $\{i,j,k\}=\{1,2,3\}$. We can
find $v_i\in\mathfrak{m}_i\backslash\{0\}$ for each $i$, such that $Q([v_1,v_2],v_3)\neq0$. By the bi-invariant
property, $Q([v_2,v_3],v_1)$ and $Q([v_3,v_1],v_2)$ are also nonzero.
For any $\theta_1,\theta_2,\theta_3>0$, we can find $\lambda_1,\lambda_2,\lambda_3>0$,
such that $u_i=\lambda_i v_i$ satisfies $\alpha_i^2(u_i)=\theta_i$ for each $i$.
By Lemma \ref{lemma-9}, there exists $u'\in\mathfrak{h}$, such that
$\tfrac{\partial L}{\partial \theta_1}[u',u_1]+(\tfrac{\partial L}{\partial \theta_3}
-\tfrac{\partial L}{\partial\theta_2})[u_2,u_3]=0$, and then
\begin{eqnarray*}
(\tfrac{\partial L}{\partial \theta_3}
-\tfrac{\partial L}{\partial\theta_2})\lambda_1\lambda_2\lambda_3 Q(v_1,[v_2,v_3])=
(\tfrac{\partial L}{\partial \theta_3}
-\tfrac{\partial L}{\partial\theta_2}) Q(u_1,[u_2,u_3])=-\tfrac{\partial L}{\partial\theta_1}Q(u_1,[u',u_1])=0.
\end{eqnarray*}
So $\tfrac{\partial L}{\partial \theta_2}=\tfrac{\partial L}{\partial\theta_3}$ at $(\theta_1,\theta_2,\theta_3)$. For the same reason,
we also have $\tfrac{\partial L}{\partial \theta_1}=\tfrac{\partial L}{\partial\theta_2}$ at $(\theta_1,\theta_2,\theta_3)$.

Now, let
\begin{equation*}
  \Omega=\{(\theta_1,\theta_2,\theta_3)|\theta_i\geq 0,i=1,2,3 \} \backslash \{(0,0,0)\},
\end{equation*}
and the interior of $\Omega$ is denoted by
\begin{equation*}
  \Omega^0=\{(\theta_1,\theta_2,\theta_3)|\theta_i> 0, i=1,2,3 \}.
\end{equation*}
To summarize, we have
\begin{equation*}
  \begin{cases}
\frac{\partial L}{\partial \theta_1}=
\frac{\partial L}{\partial \theta_2}=\frac{\partial L}{\partial \theta_3},\  \forall (\theta_1,\theta_2,\theta_3)\in \Omega^0; \\
\\
L(t\theta_1,t\theta_2,t\theta_3)=tL(\theta_1,\theta_2,\theta_3).
\end{cases}
\end{equation*}

the level sets of $L$ are the planes
$\theta_1+\theta_2+\theta_3=\mathrm{const}$. Together with the positive 1-homogeneity of $L$,
we see that $L(\theta_1,\theta_2,\theta_3)$ is a linear function on $\Omega^0$. By the smoothness of $L$ on $\Omega$, we have $L$ is a linear function on $\Omega$, that is
$L(\theta_1,\theta_2,\theta_3)=\lambda(\theta_1+\theta_2+\theta_3)$ on $\Omega$ for some $\lambda>0$. The proof of Theorem \ref{theorem 1.2} is finished.

As a direct corollary of Theorem \ref{theorem 1.2}, we have
\begin{corollary}\label{cor-1}
Any standard homogeneous $(\alpha_1,\alpha_2,\alpha_3)$-metric is g.o. on a Wallach space $SU(3)/T^2$, $Sp(3)/Sp(1)^3$ or $F_4/Spin(8)$ if and only if it is a normal homogeneous Riemannian metric.
\end{corollary}

\begin{remark}For most generalized Wallach space $G/H$ of Type II,
the three summands in its isotropy representation are pairwise non-equivalent $\mathrm{(}$see Theorem 3.18
in \cite{CKL2016}$\mathrm{)}$, and then any homogeneous $(\alpha_1,\alpha_2,\alpha_3)$-metric on $G/H$ must be a standard homogeneous $(\alpha_1,\alpha_2,\alpha_3)$-metric with respect to the canonical decomposition. On the other hand, for a generalized Wallach space $G/H$ of Type III, the $\mathrm{Ad}(H)$-action on each $\mathfrak{m}_i$ is equivalent to the adjoint representation, so there exists many more homogeneous $(\alpha_1,\alpha_2,\alpha_3)$-metrics.
\end{remark}

{\bf Acknowledgement}.
This paper is supported by Natural Science Foundation of Shandong Province (No. ZR2020QA001),
Beijing Natural Science Foundation (No. 1222003), and National Natural Science Foundation of China (No. 12131012, No. 11821101).


\bigskip

\begin{thebibliography}{99}

\bibitem{AW} A.Arvanitoyeorgos, Y.Wang, Homogeneous geodesic in generalized Wallach spaces,
The Belgian Mathematical Society 24 (2) (2017), 257-270.

\bibitem{BC}  D. Bao, S.S. Chern, Z.Shen, An Introduction to Riemann-Finsler Geometry. Springer, New
York (2000).

\bibitem{BN2020} V.N. Berestovskii and Yu.G. Nikonorov, Riemannian Manifolds and Homogeneous Geodesics, Springer Monogr. Math., Springer, Cham, 2020.

\bibitem{DA} D.V. Alekseevsky , A. Arvanitoyeorgos, Riemannian flag manifolds with homogeneous geodesics, Trans. Amer. Math.Soc. 359(8), 3769-3789 (2007).

\bibitem{DN}   D.V. Alekseevsky, Yu.G. Nikonorov, Compact Riemannian manifolds with
homogeneous geodesics, SIGMA Symmetry Integrability Geom. Methods Appl., 5, 093, 16 pages,(2009).


\bibitem{Niko1}  Yu.G. Nikonorov, Classification of generalized Wallach spaces, Geom. Dedicata, 181(1), (2016).

\bibitem{CKL2016} Z. Chen, Y. Kang and K. Liang, Invariant Einstein metrics on three-locally-symmetric spaces, Comm. Anal. Geom. 24 (4) (2016), 769-792.

\bibitem{CN} Z. Chen, Yu.G. Nikonorov, Geodesic orbit Riemannian spaces with two isotropy summands. I. Geom Dedicata. 203, 163-178, (2019).

\bibitem{D1} S. Deng, Homogeneous Finsler spaces, Springer Monographs in Mathematics, Springer, New York, 2012.

\bibitem{DH} S. Deng, Z. Hou, Weakly symmetric Finsler spaces, Commun. Contemp. Math., 12(2) , 309-323, (2010).

\bibitem{DH1} S. Deng, Z. Hou, Invariant Finsler metrics on homogeneous manifolds, J. Phys. A 37, 8245-8253, (2004).

\bibitem{DH2} S. Deng, Z. Hou, Naturally reductive homogeneous Finsler spaces, manuscripta math. 131, 215-229 (2010).

\bibitem{DK} W. Dickinson and M. Kerr, The geometry of compact homogeneous spaces with two isotropy summands, Ann. Glob. Anal. Geom. $\mathbf{34}$ , 329-350,(2008)  .

\bibitem{DX2016} S. Deng and M. Xu, $(\alpha_1,\alpha_2)$-metrics and Clifford-Wolf homogeneity, J. Geom. Anal. 26 (3) (2016), 2282-2321.


\bibitem{Gordon}  C.S. Gordon, Homogeneous Riemannian manifolds whose geodesies are orbits,155-174, Topics in Geometry: In Memory of Joseph D'Atri (Ed. S. Gindikin), Progress in Nonlinear Differential Equations 20, Birkhauser-Verlag, Boston, Basel, Berlin, (1996).

\bibitem{GN}  C.S. Gordon , Yu.G. Nikonorov, Geodesic Orbit Riemannian Structures on $\mathbf{R}^n$. J. Geom. Phy. 134, 235-243, (2018).


\bibitem{Niko} Yu.G. Nikonorov, Killing vector fields of constant length on compact homogeneous Riemannian
manifolds, Ann. Glob. Anal. Geom., 48(4), 305-330, (2015).


\bibitem{H1} L. Huang, On the fundamental equations of homogeneous Finsler spaces, Differential Geom. Appl., 40, 187-208.


\bibitem{L}  D. Latifi, Homogeneous geodesics in homogeneous Finsler spaces, J. Geom. Phys. 57, 1421-1433 (2007).

\bibitem{S}   N. P. Souris, Geodesic orbit metrics in compact homogeneous manifolds with equivalent isotropy
submodules, Transform. Groups. 23(4),  1149-1165, (2018).


\bibitem{KV}    O. Kowalski ,L. Vanhecke, Riemannian manifold with homogeneous geodesics, Boll. Unione Math. Ital. B(7) 5,189-246 (1991).



\bibitem{TX2023} J. Tan and M. Xu, Naturally reductive $(\alpha_1,\alpha_2)$ metrics,
Acta Math. Sci., 43B(4), 1547-1560, (2023).

\bibitem{X}    M Xu,  Geodesic orbit spheres and constant curvature in Finsler geometry, Diff. Geom. Appl., 61(2018), 197-206.

\bibitem{XD}   M. Xu , S. Deng, Normal homogeneous Finsler spaces, Transform. Groups, 22(4), 1143-1183, (2017).


\bibitem{XZ}  M. Xu , L. Zhang, $\delta$-homogeneity in Finsler geometry and the positive curvature problem, Osaka J. Math., 55 (1) , 177-194, (2018).

\bibitem{ZX} L. Zhang, M. Xu, Standard homogeneous $(\alpha_1,\alpha_2)$-metrics and geodesic orbit proberty, Math. Nachr 295(7), 1443-1453, (2022).



\bibitem{YD}   Z, Yan , S. Deng ,  Finsler spaces whose geodesics are orbits,  Differ. Geom. Appl.,  36, 1-23, (2014).

\bibitem{ZYD2023} S. Zhang, Z. Yan and S. Deng, Naturally reductive homogeneous $(\alpha,\beta)$ spaces, Publ. Math. Debrecen, 102(3-4), 415-427 (2023).

\end{thebibliography}
\end{document}